\RequirePackage{silence}
\documentclass[11pt,a4paper]{article}
\usepackage{amsfonts}
\usepackage{amsthm}
\usepackage{amsmath}
\usepackage{amscd}
\usepackage[latin2]{inputenc}
\usepackage{t1enc}
\usepackage[mathscr]{eucal}
\usepackage{indentfirst}
\usepackage{graphicx}
\usepackage{graphics}
\usepackage{pict2e}
\usepackage{epic}
\usepackage{mathtools}
\usepackage{amssymb}
\usepackage{wrapfig}
\usepackage{thmtools}
\usepackage{thm-restate}
\usepackage[colorlinks=true, linkcolor=black, anchorcolor=black,citecolor=black,filecolor=black,menucolor=black,runcolor=black,urlcolor=black]{hyperref}
\usepackage{cleveref}

\numberwithin{equation}{section}
\usepackage[margin= 2.2 cm, bottom=19.5mm,footskip=12mm,top=20mm]{geometry}
\usepackage{epstopdf}

\theoremstyle{plain}
\newtheorem{theorem}{Theorem}[section]
\newtheorem{lemma}[theorem]{Lemma}
\newtheorem{corollary}[theorem]{Corollary}

 \theoremstyle{definition}
\newtheorem{definition}[theorem]{Definition}

\newtheorem{?}[theorem]{Problem}

\newtheorem*{Def*}{Definition}

\title{Large induced trees in dense random graphs}
\author{%
Nemanja Dragani\'c \thanks{Department of Mathematics, ETH, Z\"urich, Switzerland. Email: \href{mailto:nemanja.draganic@math.ethz.ch} {\nolinkurl{nemanja.draganic@math.ethz.ch}}.}
  }

\begin{document}
\maketitle
\begin{abstract}
Erd\H{o}s and Palka initiated the study of the maximal size of induced trees in random graphs in 1983. They proved that for every fixed $0<p<1$ the size of a largest induced tree in $G_{n,p}$ is concentrated around $2\log_q (np)$ with high probability, where $q=(1-p)^{-1}$. De la Vega showed concentration around the same value for $p=C/n$ where $C$ is a large constant, and his proof also works for all larger $p$. We show that for any given tree $T$ with bounded maximum degree and of size $(2-o(1))\log_q(np)$, $G_{n,p}$ contains an induced copy of $T$ with high probability for $n^{-1/2}\ln^{10/9}n\leq p\leq 0.99$. This is asymptotically optimal.
\end{abstract}
\section {Introduction}

The problem of existence of large induced trees in random graphs was first studied by of Erd\H{o}s and Palka \cite{erdos1983trees}. For a given graph $G$, with $T(G)$ we denote the size of a largest induced subtree of $G$. In the mentioned paper the following was proven:
\begin{theorem}
For every $\epsilon>0$ and for every fixed $0<p<1$, with probability $1-o(1)$ it holds that
\begin{equation*}
    (2-\epsilon)\log_q (np)<T(G_{n,p})<(2+\epsilon)\log_q (np)
\end{equation*}
where $q=(1-p)^{-1}$.
\end{theorem}
The proof of the upper bound in this paper actually works for every $p=p(n)<1-\alpha$ where $\alpha>0$ is a constant. Therefore $2\log_q (np)$ is asymptotically the best we can hope for the maximal size of an induced subtree in every regime.\par
Erd\H{o}s and Palka conjectured that for $p=C/n$, for a large fixed constant $C$, $G_{n,p}$ contains a linear sized induced subtree. Several researchers independently proved the conjecture \cite{frieze1987large2,kuvcera1987large,luczak1988maximal}. After this question was settled, de la Vega \cite{de1996largest}  showed that $T(G)$ is at least $(1-o(1))2\log_q(np)$ in the same regime. In fact, his idea can be used to show the same bound for all larger $p$.

\par
The first papers which deal with sizes of specific induced trees in random graphs are by Frieze and Jackson \cite{frieze1987large} and Suen \cite{suen1992large}. They were interested in the size of the largest path in the regime $p=C/n$, where $C$ is a large constant. {\L}uczak \cite{luczak1993size} proved that for $p=C/n$ where $C$ is large enough, the size of the largest induced path in $G_{n,p}$ is at least  $(1/2-o(1))2\log_q(np)$.
In the dense case where $p$ is a constant, Ruci{\'n}ski \cite{rucinski1987induced} proved that the largest induced path is at least $(1-o(1))2\log_q(np)$, which was extended by Dutta and Subramanian \cite{dutta2018induced} to $p\geq n^{-1/2}(\ln n)^2$. Moreover \cite{dutta2018induced} shows concentration in two values for the size of the largest path.

\par In this paper we prove that for any fixed tree $T$ of size $(1-o(1))2\log_q(np)$ and bounded maximum degree and $n^{-1/2}\ln^{10/9}n\leq p\leq 0.99$, $G_{n,p}$ contains $T$ as an induced subtree, with high probability, i.e. with probability $1-o(1)$.

\begin{theorem}\label{general}
    Let $\Delta\geq 2$ be an integer and $p=p(n)$ be such that $n^{-1/2}\ln^{10/9}n<p<0.99$. Let $T$ be a tree with  $b=(1-o(1))2\log_q(np)$  vertices and maximum degree $\Delta$. Then with high probability $G_{n,p}$ contains $T$ as an induced subgraph.
\end{theorem}

More precisely, we will prove the following sharper bound, for $p<1/\ln n$.

\begin{theorem}\label{Main Thm}
Let $\Delta\geq 2$ be an integer and $p=p(n)$ be such that $n^{-1/2}\ln^{10/9}n<p<\frac{1}{\ln n}$. Let $T$ be a tree with 
$$ b=\Big(1-\frac{3\ln\ln (np)}{2\ln (np)}\Big)2\log_q(np)$$ vertices and maximum degree $\Delta$. Then with high probability $G_{n,p}$ contains $T$ as an induced subgraph.
\end{theorem}
To get the remaining part of Theorem \ref{general}, when $p\geq 1/\ln n$, one can basically follow the same proof, but the calculations get much simpler. For clarity of presentation and to avoid repeating arguments we omit this part of the proof.

It is known that the size of the largest independent set in $G(n,p)$ is with high probability asymptotically equal to $2\log_q(np)$, see for example \cite{janson2011random}. As a natural bridge between trees and independent sets, we further propose the study of large induced forests. One can easily see that Theorem \ref{general} yields the following result:
\begin{theorem}\label{Thm:3}
    Let $\Delta\geq 2$ be an integer and $\varepsilon>0$. Let $p=p(n)$ be such that $n^{-1/2}\ln^{10/9}n<p<0.99$. Let $F$ be a forest with 
$$b=(1-\varepsilon)2\log_q(np)$$ vertices and maximum degree $\Delta$. Then with high probability $G_{n,p}$ contains $F$ as an induced subgraph.
\end{theorem}

\begin{proof}
 Since $F$ has $(1-\varepsilon)2\log_q(np)$ vertices, it also has at most this many components. We construct a tree $T$ as follows: We take a path of length $\varepsilon\log_q(np)$ and to each vertex we attach with additional edges a nearly equal number of components of $F$, meaning that each vertex in the path will have degree at most $\lceil \frac{2(1-\varepsilon)}{\varepsilon}\rceil+2$. Now, since $T$ has bounded degree and $(1-\varepsilon/2)2\log_q(np)$ vertices, we use Theorem \ref{general} to get that there is an induced copy of $T$ in $G_{n,p}$ with high probability, but since $T$ contains $F$ as an induced subgraph, we are done.
\end{proof}
Forests which are also of particular interest are matchings. In the case where $p$ is constant, it is known that the number of vertices of a largest induced matching in $G_{n,p}$ is concentrated in two values which are asymptotically equal to $2\log_q(np)$ \cite{clark2001strong}. Theorem \ref{Thm:3} shows that a largest induced matching of $G_{n,p}$ has asymptotically at least $2\log_q(np)$ vertices, when $n^{-1/2}\ln^{10/9}n<p<0.99$.

\section{Proof of Theorem \ref{Main Thm}}

Let $c=c(n)=pn$ and $h=\frac{3\ln\ln (np)}{\ln (np)}$, so that $b=(2-h)\log_q(np)$. 

We will actually prove the theorem for $b:=(2-h)\frac{\ln c}{c}n$. This implies the validity of the theorem in its original form, as $(2-h)\frac{\ln c}{c}n>(2-h)\log_q(np)$ because:
\begin{align*}
    \log_q(np)=\frac{\ln c}{\ln (1-p)^{-1}}=\frac{\ln c}{-\ln (1-c/n)}< \frac{\ln c}{c}n,
\end{align*}
 which means that we can even have a larger tree\footnote{For the proof of Theorem \ref{general} when $p$ is constant, one has to work with $\log_q(np)$, as this is asymptotically not the same as $\frac{\ln c}{c}n$.}.

The proof technique we are going to use is the second moment method. For smaller $p$ it is known that the problem of determining the largest independent set in the random graph is not approachable by vanilla second moment calculations, as the lower bound on the probability of the existence of independent sets of size close to $2\log_q(np)$ is exponentially small. Nevertheless, Dani and Moore \cite{dani2011independent} determined the asymptotically optimal result for large independent sets, by using the second moment in a clever non-standard way, by assigning weights to the edges. It might be that their approach could yield results similar to what we prove, but for smaller $p$.
\subsection{The Second Moment Method}

Label the vertices of $T$ with numbers from $1$ to $b$.
Also label the vertices of $G_{n,p}$ with numbers from $1$ to $n$.
In the text that follows we will identify the vertices of the graphs with their labels.\\
Let $L=\{\Phi_i:i\in I\}$ be the set of all injective functions from $\{1,2,...,b\}$ to $\{1,2,...,n\}$. Let $X_i$ be the indicator random variable for the event that the ordered set of vertices $(\Phi_i(1),...,\Phi_i(b))$ induces a tree isomorphic to $T$, where the isomorphism preserves the order of the vertices. In other words, this is the event that $\{\Phi_i(u),\Phi_i(v)\}$ is an edge in $G_{n,p}$ if and only if $\{u,v\}$ is an edge in $T$, for every two vertices $u,v$ in $T$. 
Let $A_i=Im(\Phi_i)$, meaning that $A_i$ is the set of vertices in the image of $\Phi_i$. We also define $T_i$ to be the tree induced by vertices in $A_i$ if $X_i=1$.\\

Let $X$ be the sum of all $X_i$. Our goal is to prove that $X>0$ with probability $1-o(1)$ when n tends to infinity. In order to do this we will use Chebyshev's Inequality:
\begin{equation*}
    Pr(X=0)\leq \frac{Var(X)}{E[X]^2}\leq \frac{1}{E[X]}+
    \sum_i\sum_{j:2\leq |A_i\cap A_j|\leq b}\frac{E[X_iX_j]-E[X_i]E[X_j]}{E[X]^2}.
\end{equation*}
Note that if $|A_i\cap A_j|<2$ then $X_i$ and $X_j$ are independent, and therefore $E[X_iX_j]-E[X_i]E[X_j]=0$.\\
Note also that 

  \begin{align*}
       \sum_i\sum_{j:2\leq |A_i\cap A_j|\leq b}&\frac{E[X_iX_j]-E[X_i]E[X_j]}{E[X]^2}=  \\
       &\sum_i\sum_{j:2\leq |A_1\cap A_j|\leq b}\frac{E[X_1X_j]-E[X_1]E[X_j]}{E[X]^2}.
  \end{align*} 
  Since $E[X_1X_j]=E[X_j|X_1=1]E[X_1]$ we get that:
  \begin{align*}
      Pr(X=0)&\leq \frac{1}{E[X]}+\sum_i\sum_{j:2\leq |A_1\cap A_j|\leq b}\frac{E[X_j|X_1=1]E[X_1]-E[X_1]E[X_j]}{E[X]^2}\\
      &=\frac{1}{E[X]}+\sum_{j:2\leq |A_1\cap A_j|\leq b}\sum_i\frac{E[X_1](E[X_j|X_1=1]-E[X_j])}{E[X]^2}\\
      &=\frac{1}{E[X]}+\sum_{\ell=2}^{b}\sum_{|A_1\cap A_j|=\ell}\frac{E[X_j|X_1=1]-E[X_j]}{E[X]}.
  \end{align*}

In the rest of the paper we will prove that the last expression tends to $0$ as $n$ tends to infinity.\\

Observe that from linearity of expectation we have
\begin{equation}\label{Eq:1}
E[X]=\frac{n!}{(n-b)!}p^{b-1}(1-p)^{\binom{b-1}{2}},
\end{equation}
as there are $b-1$ edges that have to be present in a tree induced by a fixed ordered set of vertices, and the rest have to be non-edges. In Section 4 we prove that $E[X]$ tends to infinity.
\newtheorem{Claim}{Claim}[section]
\begin{Claim}\label{Cl:1}
$\lim_{n\rightarrow \infty}E[X]=\infty$.
\end{Claim}

Now it is clear that if we prove $$
H=\sum_{\ell=2}^{b}\sum_{|A_1\cap A_j|=\ell} \frac{E[X_j|X_1=1]-E[X_j]}{E[X]}=o(1)$$
we would be done. We will actually prove that 
\begin{equation}\label{Eq:2}
H\leq H_1:=\sum_{\ell=2}^{b}\sum_{|A_1\cap A_j|=\ell} \frac{E[X_j|X_1=1]}{E[X]}=o(1).
\end{equation}

\subsection{The Counting Framework}
In this subsection we will classify $\Phi_i$'s for which the term ${E[X_j|X_1=1]}$ in the last expression is non-zero.

\begin{definition}
Let $i\in I$ and $A=A_1\cap A_i$. If it holds that $\{\Phi^{-1}_1(x),\Phi^{-1}_1(y)\}$ is an edge in $T$ if and only if $\{\Phi^{-1}_i(x),\Phi^{-1}_i(y)\}$ is an edge in $T$ for every two $x,y\in A$, then  $\Phi_i$ is \emph{compatible} with  $\Phi_1$.
\end{definition}
Note that if
$\Phi_i$ and $\Phi_1$ are not compatible, then $E[X_i|X_1=1]=0$. 

\begin{definition}
Define $S(\ell,k)$ to be the number of functions $\Phi_i\in L$ compatible with $\Phi_1$ such that $A=A_1\cap A_i$ has $\ell$ vertices and the subgraph of $T$ induced by vertices in $\Phi^{-1}[A]$ consists of a forest with $k$ connected components.
\end{definition}

For each of the mentioned $S(\ell,k)$ functions $\Phi_j$, we have that the following holds
\begin{equation}\label{Eq:3}
    E[X_j|X_1=1]=Pr[X_j=1|X_1=1]=p^{b-1-(\ell-k)}(1-p)^
    {\binom{b-1}{2}-{\binom {\ell}{2}}+(\ell-k)}.
\end{equation}
This is easy to see as we just have to have in mind that the vertices in $A$ already span a fixed graph, because $X_1=1$.\par
One simple observation of cardinal importance is the following. For a fixed $\ell<b$, if $k>(b-\ell)d$ then $S(\ell,k)= 0$. To see this, suppose that there exists at least one function $\Phi_i$ which would be a counterexample. Observe that each of the $k$ mentioned connected components in the graph is connected with an edge to at least one vertex of the remaining $b-\ell$ vertices in $T$. Since each vertex has degree at most $d$ there can be at most $(b-\ell)d$ of those components, and therefore $k\leq(b-\ell)d$. Also note that the number of components cannot be larger than the number of vertices $l$. \\In the case when $b=\ell$, we have that $k=1$. Otherwise $S(b,k)=0$.\\
\par Define 

\[ k_{\ell}=\begin{cases} 
      1 & \ell=b \\
      \min\{\ell,(b-\ell)d\} & \ell<b \\
   \end{cases}
\]
From the previous discussion we get that $S(\ell,k)=0$ whenever $k>k_{\ell}$.
Combining (\ref{Eq:1}),(\ref{Eq:2}) and (\ref{Eq:3}) we get that it is enough to show that $\Tilde{H}=o(1)$ where
\begin{equation}\label{Eq:4}
    H_1\leq \Tilde{H}:=\sum_{\ell=2}^{b}\sum_{k=1}^{k_{\ell}}\frac{S(\ell,k)}{(n)_b}p^{-(\ell-k)}(1-p)^{-{\ell \choose 2}+\ell-k},
\end{equation}
where we use the notation $(r)_t={r \choose t}t!$. Next we want to estimate $S(\ell,k)$.

\subsection{Estimating $S(\ell,k)$}
\par
We will prove that
$$S(\ell,k)<(n-b)_{b-\ell}{b \choose k}^2k!{k+\ell-1 \choose \ell}d^{\ell},$$
where $d=2^{2\Delta}\Delta!.$\\

\par 
 In order to choose a compatible function $\Phi_i$ for which $A=A_1\cap A_i$ is of size $\ell$ and the subgraph of $T$ induced by $\Phi_i^{-1}[A]$ has $k$ connected components, we first choose one vertex for each of the $k$ components which has the lowest label in $T$ -- we call them \textit{roots}. This can be done in ${b \choose k}$ ways. The total number of ways to decide where each of these vertices is mapped by $\Phi_i$ is ${b \choose k}k!$. So the number of ways to choose the roots and to choose where they are mapped is ${b \choose k}^2k!$. It is important to remember that $A_1$ is fixed.
\par
Now we count the number of ways to choose and map the remaining vertices in $\Phi_i^{-1}[A]$. \\
First note that there are less than ${k+\ell-1 \choose \ell}$ ways to choose how many vertices each of the $k$ components will have.
 Each vertex in $T$ has at most $\Delta$ neighbours, so for each root there are at most $2^\Delta$ ways to choose which neighbours of the root in $T$ will map to $A$. There are also at most $2^\Delta$ ways to choose the neighbours of a root\footnote{Recall that $T_i$ is the tree induced by vertices in $A_i$ if $X_i=1$.} in $T_1$ and at most $\Delta!$ ways to choose which neighbours correspond to each other. Now by induction, we get that for a component in $T$ of size $s$ and a fixed root, there are at most $(2^{2\Delta}\Delta!)^s$ choices. Therefore, the upper bound on the total number of ways to choose the vertices in the intersection is $(2^{2\Delta}\Delta!)^\ell$.
\par 
To complete the choice of $\Phi_i$, we have to map the remaining $b-\ell$ vertices from $T$. This can be done in $(n-b)_{b-\ell}$ ways. From the discussion above we get that $S(\ell,k)$ is bounded by the product
\begin{equation*}
    S(\ell,k)\leq(n-b)_{b-\ell}{b \choose k}^2k!{k+\ell-1 \choose \ell}(2^{2\Delta}\Delta!)^\ell.
\end{equation*}
Now we get that $\Tilde{H}$ from Equation (\ref{Eq:4}) is upper bounded by
\begin{equation}\label{Eq:5}
 \sum_{\ell=2}^{b} \frac{(n-b)_{b-\ell}}{(n)_{b}}p^{-\ell}(1-p)^{\ell-{\ell \choose 2}}\sum_{k=1}^{k_{\ell}}{b \choose k}^2k!{k+\ell-1 \choose \ell}d^{\ell}\Big(\frac{p}{1-p}\Big)^k.
\end{equation}

To summarize: we used Chebyshev's inequality to give an upper bound on the probability that $T$ is an induced subgraph in $G_{n,p}$. The trick was to carefully count the number of ways in which two trees can intersect if they are induced subgraphs. Now it is enough to prove that expression (\ref{Eq:5}) tends to $0$ when $n$ goes to infinity. We prove this in the next section.

\section{Proof of convergence of expression (\ref{Eq:5})}
First we will state some auxiliary results, some of which we prove in Section $4$. We write the proofs separately, as these simple calculus exercises would otherwise interfere with the natural flow of this paper (although we admit that the rest of the proof consists only of tedious calculations).
\subsection{Auxiliary Statements}
First we will examine the second sum in expression (\ref{Eq:5}). Let 
\begin{equation}
    f(k)=f_\ell(k):={b \choose k}^2k!{k+\ell-1 \choose \ell}d^{\ell}\Big(\frac{p}{1-p}\Big)^k.
\end{equation}

\begin{Claim}[2]\label{Cl:2}
    The function $f(k)$ is maximized for $k=k_{\ell}$ when $1\leq k \leq k_{\ell}$. 
\end{Claim}

This brings us to the wanted simplification of the second sum:

\begin{lemma}\label{Lem:1}
    The following holds   $$\sum_{k=1}^{k_{\ell}}f(k)<k_{\ell}f(k_{\ell})$$.
\end{lemma}
\noindent Now we turn to the first part of formula (\ref{Eq:5}).
\begin{Claim}[3]\label{Cl:3}
For $n$ large enough it holds that :
    $$\frac{(n-b)_{b-\ell}}{(n)_{b}}< \frac{6^\ell}{n^\ell}e^{-b^2/n}<\Big(\frac{6}{n}\Big)^\ell$$.
\end{Claim}

We continue with another simple lemma.
\begin{lemma}\label{Lem:2}
    For all integers $\ell$ where $1\leq \ell \leq b$ it holds that
    $$\frac{(n-b)_{b-\ell}}{(n)_{b}}p^{-\ell}(1-p)^{\ell-{\ell \choose 2}}<6^lc^{-\ell+(1+1/\ln n)c{{\ell \choose 2}}/(n\ln c)}.$$
\end{lemma}
\begin{proof}
Let us estimate $p^{-\ell}(1-p)^{\ell-{\ell \choose 2}}$. 
\begin{align*}
    p^{-\ell}(1-p)^{\ell-{\ell \choose 2}}&<\Big(\frac{n}{c}\Big)^\ell(1-c/n)^{-{\ell \choose 2}}\\
    &<\Big(\frac{n}{c}\Big)^\ell e^{{\ell \choose 2}(\frac{c}{n}+\frac{c^2}{n^2})}\\
    &=\Big(\frac{n}{c}\Big)^\ell c^{{\ell \choose 2}(1+c/n)c{}/(n\ln c)}.\\
    &\leq  \Big(\frac{n}{c}\Big)^\ell c^{{\ell \choose 2}(1+1/\ln n)c/(n\ln c)}.
\end{align*}
The statement follows from Claim \ref{Cl:3}.
\end{proof}

\begin{corollary} \label{Co:1}
    For all integers $\ell$ where $1\leq \ell \leq b$ it holds that
    $$\frac{(n-b)_{b-\ell}}{(n)_{b}}p^{-\ell}(1-p)^{\ell-{\ell \choose 2}}<6^lc^{-(1-o(1))hl/2}.$$
\end{corollary}
\begin{proof}
    The proof follows from the previous lemma and the fact that\\ $\ell\leq b=(2-h)n\ln c/c$. To see this we just bound the exponent in the bound from the lemma.
    \begin{align*}
        &-\ell+(1+1/\ln n)c{{\ell \choose 2}}/(n\ln c)\\
        <&-\ell+(1+1/\ln n)cb\ell/(2n\ln c)\\
        =&-\ell + (1+1/\ln n)(1-h/2)\ell\\
        =&-hl/2+\frac{1-h/2}{\ln n}\ell\\
        =&-hl/2 (1-o(1)).
    \end{align*}
\end{proof}
\subsection{Finishing the proof}
\begin{lemma}
The expression 
\begin{equation*}
 \Tilde{H}=\sum_{\ell=2}^{b} \frac{(n-b)_{b-\ell}}{(n)_{b}}p^{-\ell}(1-p)^{\ell-{\ell \choose 2}}\sum_{k=1}^{k_{\ell}}{b \choose k}^2k!{k+\ell-1 \choose \ell}d^{\ell}\Big(\frac{p}{1-p}\Big)^k
\end{equation*}
converges to $0$ as n tends to infinity.
\end{lemma}
\begin{proof}
In order to prove the lemma and finish the proof we will show that for each $\ell$ the summands are so small so that the whole sum will be $o(1)$.\\

With $g(\ell)$ we denote the $\ell$'th summand:
\begin{equation}
    g(\ell)=\frac{(n-b)_{b-\ell}}{(n)_{b}}p^{-\ell}(1-p)^{\ell-{\ell \choose 2}}\sum_{k=1}^{k_{\ell}}{b \choose k}^2k!{k+\ell-1 \choose \ell}d^{\ell}\Big(\frac{p}{1-p}\Big)^k.
\end{equation}
We will make a case distinction depending on what $k_{\ell}$ is.\\[5pt]
\textbf{Case I:} $k_{\ell}=(b-\ell)d$\\[5pt]

Using Lemma \ref{Lem:1} together with Corollary \ref{Co:1} we get an upper bound on the $\ell$'th summand:
\begin{equation}\label{Eq:8}
    g(\ell) <6^\ell c^{-(1-o(1))hl/2}k_{\ell} f(k_{\ell}).
\end{equation}
We start with the following observation (since $k_{\ell}\leq \ell$):\\
\begin{equation*}
    \ell\geq k_{\ell}=(b-\ell)d\implies \ell\geq \frac{d}{d+1}b.
\end{equation*}
Let $\ell=(1-\varepsilon)b$ for some $\varepsilon\leq\frac{1}{d+1}$. From this we get $k_{\ell}=\varepsilon bd$.
Using standard upper bounds for factorials and binomial coefficients we get from (\ref{Eq:8}):
\begin{align*}
   g(\ell)&<6^{(1-\varepsilon)b}c^{-(1-o(1))h(1-\varepsilon)b/2}k_\ell
    \Big(\frac{be}{k_{\ell}}\Big)^{2k_{\ell}}\Big(\frac{k_{\ell}}{e}\Big)^{k_{\ell}} \sqrt{4k_{\ell}\pi} 2^{2l}d^{\ell}p^{k_{\ell}}\frac{1}{(1-p)^{k_{\ell}}}.
\end{align*}
    After rearranging the terms and the use of $(1-p)^{-1}<2$ we get:
   \begin{align*}
       g(\ell)&<6^{(1-\varepsilon)b}c^{-(1-o(1))h(1-\varepsilon)b/2}k_\ell
    e^{k_{\ell}}\Big(\frac{b}{k_{\ell}}\Big)^{2k_{\ell}}k_{\ell}^{k_{\ell}} \sqrt{4k_{\ell}\pi} 2^{2l}d^{\ell}\Big(\frac{c}{n}\Big)^{k_{\ell}}2^{k_{\ell}}\\
    &=\Big(6^{(1-\varepsilon)b}k_\ell e^{k_\ell}\sqrt{4k_{\ell}\pi} 2^{2l}d^{\ell}2^{k_{\ell}}\Big)\Big(c^{-(1-o(1))h(1-\varepsilon)b/2}\Big(\frac{b}{k_{\ell}}\Big)^{2k_{\ell}}k_{\ell}^{k_{\ell}}\Big(\frac{c}{n}\Big)^{k_{\ell}}\Big)
   \end{align*} 
  Let $Q=Q(d)$ be a sufficiently large constant depending only on $d$ such that the expression in the first bracket is less than $Q^b$. It is easy to see that $Q$ exists, as $k_\ell,l\leq b$. This leads us to 
\begin{align*}
    g(\ell)<Q^bc^{-(1-o(1))h(1-\varepsilon)b/2}\Big(\frac{b}{k_{\ell}}\Big)^{2k_{\ell}}k_{\ell}^{k_{\ell}} \Big(\frac{c}{n}\Big)^{k_{\ell}}.
\end{align*} 
 Now we just plug in $k_{\ell}=\varepsilon bd$:
\begin{align*}
     g(\ell)&< Q^b c^{-(1-o(1))h(1-\varepsilon)b/2}\Big(\frac{1}{\varepsilon d}\Big)^{2\varepsilon bd}\Big(\frac{c\varepsilon bd}{n}\Big)^{\varepsilon bd}.
\end{align*}
Now note that $\Big(\frac{1}{\varepsilon }\Big)^\varepsilon<e$ (take logarithm of both sides). Therefore the third term in the last expression is also bounded by some function of $d$ to the power of $b$. Also note that $\varepsilon d<1$ and $b<2\frac{\ln c}{c}n$, so we get:
\begin{align*}
    g(\ell)<\Tilde{Q}^b c^{-(1-o(1))h(1-\varepsilon)b/2}(\ln c)^b.
\end{align*}
Now we plug in $h=3\frac{\ln\ln c}{\ln c}$ and get:
\begin{align*}
    g(\ell)<\Tilde{Q}^b (\ln c)^{-3(1-o(1))(1-\varepsilon)b/2} (\ln c)^b,
\end{align*}
where again $\Tilde{Q}=\Tilde{Q}(d)$ is a function depending only on $d$.
The middle term converges to zero much faster than the others go to infinity, so for $c$ large enough $g(\ell)$ goes to $0$ exponentially fast in terms of $b$, where the base of the exponential is of order $\ln c$. Therefore $g(\ell)=o(b^{-1})$ and since there are at most $b$ summands in this case, we get that this part of the sum contributes $o(1)$ to the whole sum.\\\\
\textbf{Case II:} $k_{\ell}=\ell$ and $\ell<b$\\
In this case we get from Lemma 1 and Lemma 2 that:
$$g(\ell)<6^\ell c^{-\ell+(1+1/\ln n)c{{\ell \choose 2}}/(n\ln c)}\ell {b \choose \ell}^2 \ell!
{\ell+\ell-1 \choose \ell}d^{\ell}p^\ell\Big(\frac{1}{1-p}\Big)^\ell$$
Let $Q$ be a sufficiently large constant depending only on $d$. Using the known estimates for binomial coefficients and factorials we get:
\begin{align*}
    g(\ell)&<6^\ell c^{-\ell+(1+1/\ln n)c{{\ell \choose 2}}/(n\ln c)} \ell \Big(\frac{be}{\ell}\Big)^{2l}\Big(\frac{\ell}{e}\Big)^{\ell}\sqrt{4l\pi } 2^\ell d^{\ell}\Big(\frac{c}{n}\Big)^l2^\ell\\
    &< Q^lc^{-\ell+(1+1/\ln n)c{{\ell \choose 2}}/(n\ln c)}\Big(\frac{b^2}{\ell}\Big)^\ell\Big(\frac{c}{n}\Big)^\ell.
\end{align*}
Now let $\ell=c^{-t} b$ for some $t$ where $0<t<\frac{\ln b}{\ln c}$. Using the bounds $c\geq n^{1/2}\ln^{10/9} n$ and $b<2\frac{\ln c}{c}n$ we get that $0<t<1-1.01\frac{\ln\ln c}{\ln c}$. Plugging this in gives:
\begin{align*}
    g(\ell)&<Q^\ell c^{-\ell+(1+1/\ln n)c{{\ell \choose 2}}/(n\ln c)}\Big(\frac{bc}{nc^{-t}}\Big)^{\ell}\\
    &<Q^\ell c^{-\ell+(1+1/\ln n)c{{\ell \choose 2}}/(n\ln c)}
    \Big(\frac{2\ln c}{c^{-t}}\Big)^{\ell}\\
    &=(2Q)^\ell(\ln c)^lc^{-\ell+(1+1/\ln n)c{{\ell \choose 2}}/(n\ln c)}
    c^{t\ell}\\
    &=(2Q)^\ell(\ln c)^lc^{-\ell+(1+1/\ln n)c{{\ell \choose 2}}/(n\ln c)+t\ell}.
\end{align*}

We will prove that: 
\begin{equation}\label{Eq:9}
    c^{-\ell+(1+1/\ln n)c{{\ell \choose 2}}/(n \ln c)+t\ell}<(\ln c)^{-1.001\ell}.
\end{equation}

This will be the end of this part of the proof, as $c$ is a function of $n$ going to infinity, so our sum is bounded by a geometric series which obviously converges to 0.

\par Now we prove that (\ref{Eq:9}) holds.
Denote the exponent on the LHS of (\ref{Eq:9}) with $A$. We have to prove that 
\begin{align*}
    &c^A<(\ln c)^{-1.001\ell}\\
    \iff & A\ln c<-1.001\ell\ln\ln c\\
    \iff & A<\frac{-1.001\ell\ln\ln c}{\ln c}.
\end{align*}

In order to do this, we first transform A:
\begin{align*}
    A=&-c^{-t}b+(1+1/\ln n)c\frac{c^{-t}b(\ell-1)}{2n\ln c}+tc^{-t}b\\
    <&-c^{-t}b+(1+1/\ln n)\frac{(2-h)c^{-t}\ell}{2}+tc^{-t}b\\
    =&-b\Big(c^{-t}-(1+1/\ln n)\frac{(2-h)c^{-2t}}{2}-tc^{-t}\Big)\\
    =&-\ell\Big(1-(1+1/\ln n)\frac{(2-h)c^{-t}}{2}-t\Big).
\end{align*}
So if we prove that 
\begin{equation}\label{Eq:10}
	1-(1+1/\ln n)\frac{2-h}{2}c^{-t}-t>\frac{1.001\ln\ln c}{\ln c}
\end{equation}
for $0<t<1-1.01\frac{\ln\ln c}{\ln c}$ we would be done. \\
Let $x=t-1$. Now (\ref{Eq:10}) transforms into
\begin{equation*}
 -x-(1+1/\ln n)(1-h/2)c^{-x-1}>\frac{1.001\ln\ln c}{\ln c}
\end{equation*}
for $-1<x<-1.01\frac{\ln\ln c}{\ln c}$. \\

We prove this inequality by analyzing the function
\begin{equation*}
    r(x)=-x-(1+1/\ln n)(1-h/2){c^{-x-1}},
\end{equation*}
where $c,n$ and $h$ are fixed. Its derivative is 
\begin{equation*}
    r'(x)=-1+(1+1/\ln n)\ln c(1-h/2){c^{-x-1}}.
\end{equation*}
Note that $r'(x)$ is a decreasing function on $\mathbb{R}$ and it has one zero. This means that $r(x)$ is increasing until a certain point and then decreases. Therefore $r(x)$ restricted to $(-1,-1.01\frac{\ln\ln c}{\ln c})$ attains its minimum in one of the boundary points of the interval, which in this case is the point $-1.01\frac{\ln\ln c}{\ln c}$. We conclude:
\begin{equation*}
   -x-(1+1/\ln n)(1-h/2)c^{-x-1}\geq  r(-1.01\frac{\ln\ln c}{\ln c})>   \frac{1.001\ln\ln c}{\ln c},
\end{equation*}

so we are done.\\[5 pt]
\textbf{Case III:} $k_{\ell}=1$ and $l=b$\\
Like in Case I we get:
\begin{align*}
    g(\ell)&<6^\ell c^{-(1-o(1))h\ell/2}f(k_{\ell})\\
&=6^b c^{-(1-o(1))hb/2}
{b \choose 1}^2{b\choose b}d^{b}\Big(\frac{p}{1-p}\Big),
\end{align*}
which obviously tends to zero, as the second term dominates in the expression like in Case I.
\end{proof}

\section{Proofs of Claims}

\begin{proof}(of Claim \ref{Cl:1})

 We estimate $E[X]$ using Stirling's approximation:
\begin{align*}
    E[X]&\sim \frac{(\frac{n}{e})^n\sqrt{2n\pi}}
    {(\frac{n-b}{e})^{n-b}\sqrt{2(n-b)\pi}}\Big(\frac{c}{n}\Big)^{b-1}(1-c/n)^{\frac{(b-1)(b-2)}{2}}\\
    &> (\frac{n}{n-b})^{n-b}(n/e)^b \Big(\frac{c}{n}\Big)^{b}(1-c/n)^{\frac{b^2}{2}}\\
     &\sim \Big(1+\frac{b}{n-b}\Big)^{n-b}(n/e)^b \Big(\frac{c}{n}\Big)^{b}e^{-\frac{cb^2}{2n}}\\
\end{align*}
 (and after using $(1+x)>e^{x-x^2}$)
 \begin{align*}
     &\sim e^{((n-b)\frac{b}{n-b}(1-\frac{b}{n-b}))}(n/e)^b \Big(\frac{c}{n}\Big)^{b}e^{-\frac{cb^2}{2n}}\\
     &= e^{b(1-\frac{b}{n-b})} \Big(\frac{c}{e}\Big)^{b}e^{-\frac{b(2-h)\ln c}{2}}\\
     &=e^{-b\frac{b}{n-b}} c^{b}c^{-\frac{b(2-h)}{2}}\\
     &=e^{-\frac{b^2}{n-b}}c^{hb/2}\\
     &=e^{-\frac{b^2}{n-b}}(\ln c)^{3b/2}\rightarrow \infty.
 \end{align*}
\end{proof}

\begin{proof}(of Claim \ref{Cl:2})

When $k_{\ell}=1$ this is trivial. Assume $k_{\ell}>1$.
\par
We will prove that the ratio of the outputs of the function for two consecutive values is greater than 1.

\begin{align*}
    \frac{f(k+1)}{f(k)} &=\frac{ (\frac{b!}{(k+1)!(b-k-1)!})^2 (k+1)!\frac{(k+\ell)!}{\ell!k!}d^{\ell}\Big(\frac{p}{1-p}\Big)^{k+1}}
    { (\frac{b!}{(k)!(b-k)!})^2 (k)!\frac{(k+\ell-1)!}{\ell!(k-1)!}d^{\ell}\Big(\frac{p}{1-p}\Big)^{k}}\\[10 pt]
    &=(\frac{b-k}{k+1})^2(k+1)\frac{k+\ell}{k}\frac{p}{1-p}\\
    &\geq \frac{(b-k)^2}{k}\frac{c}{n}.
\end{align*}
Now let $k=\varepsilon b$. We get
\begin{align*}
    \frac{f(k+1)}{f(k)} &\geq \frac{(1-\varepsilon)^2b}{\varepsilon}\frac{c}{n}\\
    &\geq \frac{(1-\varepsilon)^2}{\varepsilon}\frac{n\ln c}{c}\frac{c}{n}\\
    &=\frac{(1-\varepsilon)^2}{\varepsilon}\ln c.
\end{align*}
Suppose now that 
\begin{align}
    & \frac{(1-\varepsilon)^2}{\varepsilon}\ln c\leq1 \nonumber\\
    \implies& (1-\varepsilon)^2\leq\frac{1}{\ln c} \nonumber\\
    \implies& \varepsilon \geq 1-\sqrt{\frac{1}{\ln c}}.
\end{align}
We will show that this is not possible if we set $c$ large enough in the beginning.\\
To see this, we will look at our bound on $k$:
\begin{align*}
    &k\leq (b-\ell)d\\
    \implies& k\leq(b-k)d\\
    \implies& k\frac{d+1}{d}\leq b\\
    \implies& k\leq \frac{d}{d+1}b\\
    \implies& \varepsilon\leq \frac{d}{d+1}.
\end{align*}
From the two bounds for $\varepsilon$ we get:
\begin{equation*}
    1-\sqrt{\frac{1}{\ln c}}\leq\frac{d}{d+1}
\end{equation*}
which is not true for $c$ large enough. Therefore $f$ is increasing and is therefore maximized in $k_l$.
\end{proof}

\begin{proof}(of Claim \ref{Cl:3})

We will use Stirling's approximation for factoriels:
\begin{align*}
&\frac{(n-b)_{b-\ell}}{(n)_{b}}
=\frac{\frac{(n-b)!}{(n-2b+\ell)!}}{\frac{n!}{(n-b)!}}   
=\frac{(n-b)!^2}{n!(n-2b+\ell)!}\sim\frac{(\frac{n-b}{e})^{2(n-b)}}{(\frac{n}{e})^n(\frac{n-2b+\ell}{e})^{n-2b+\ell}}\\
&=e^\ell\Big(\frac{n-b}{n}\Big)^n \Big(\frac{n-b}{n-2b+\ell}\Big)^{n-2b} \Big(\frac{1}{n-2b+\ell}\Big)^\ell\\
&<e^\ell\Big(1-\frac{b}{n}\Big)^n\Big(\frac{n-b}{n-2b}\Big)^{n-2b}\Big(\frac{2}{n}\Big)^\ell\\
\end{align*}
(after regrouping some terms we get)
\begin{align*}
&=\Big(\frac{2e}{n}\Big)^\ell\Big(1-\frac{b}{n}\Big)^{2b}\Big(\Big(1-\frac{b}{n}\Big)\frac{n-b}{n-2b}\Big)^{n-2b}\\
&<\Big(\frac{2e}{n}\Big)^\ell e^{-2b^2/n} \Big(\frac{n-2b+b^2/n}{n-2b}\Big)^{n-2b}\\
&=\Big(\frac{2e}{n}\Big)^\ell e^{-2b^2/n} \Big(1+\frac{b^2/n}{n-2b}\Big)^{n-2b}\\
&<\Big(\frac{2e}{n}\Big)^\ell e^{-2b^2/n} e^{b^2/n}\\
&=\Big(\frac{2e}{n}\Big)^\ell e^{-b^2/n}\\
&<\Big(\frac{6}{n}\Big)^\ell e^{-b^2/n}.
\end{align*}
With the last step we assure that the inequality from the statement always holds, for $n$ large enough.
\end{proof}

\section{Concluding Remarks}
\newtheorem{question}{Question}
\newtheorem{conjecture}{Conjecture}
We proved that $G(n,p)$ with high probability contains any fixed induced tree of size $b=(1-o(1))2\log_q(np)$ and with constant maximum degree for $n^{-1/2}\ln^{10/9}n\leq p\leq 0.99$. The same question is still open for smaller $p$. Accordingly, we state the following conjecture. 

\begin{conjecture}\label{Con:1}
Let $\Delta\geq 2$ be an integer and $p=p(n)$ a function such that $\frac{C}{n}<p<0.99$ where $C=C(\Delta)$ is a sufficiently large constant depending only on $\Delta$.  Let $T$ be a tree with 
$b=(1-o(1))2\log_q(np)$ vertices and maximum degree $\Delta$. Then with high probability $G_{n,p}$ contains $T$ as an induced subgraph.
\end{conjecture}

\par In \cite{dutta2018induced} it is proven that for $p>2/n$ and for any tree $T$ of size at least $2\log_q(c)+3$, $G_{n,p}$ does not contain $T$ as an induced subgraph with probability $1-o(1)$. Furthermore, if $p\geq n^{-1/2}(\ln n)^2$, they showed that the size of the largest induced path is concentrated in two values. 
\begin{question}
Can we show similar concentration bounds for arbitrary trees of bounded degree for $\frac{C}{n}<p<0.99$ ?
\end{question}

\par Theorem \ref{Thm:3} does not give the most precise $b$ we can hope for. The power of the method is limited by the condition of Theorem \ref{general} that the tree must have bounded degree. Theorem \ref{general} probably still holds if instead having a constant maximum degree, we allow the degree to tend to infinity in terms of $n$. As a consequence of this, we could also strengthen Theorem \ref{Thm:3}.
\begin{question}
How do Theorem \ref{general} and Theorem \ref{Main Thm} change if we loosen the condition that the maximum degree of the tree is bounded?
\end{question}
\par As mentioned in the introduction, the existence of large induced matchings is a corollary of Theorem \ref{Thm:3}. It would be of particular interest to determine the minimal number of induced matchings which cover $G_{n,p}$ with high probability:   
\begin{question}
What is the minimal number of induced matchings needed to cover $G_{n,p}$ with high probability?
\end{question}

\end{document}